\documentclass[11pt]{amsart}

\usepackage{amsmath,amssymb,amsthm}
\usepackage{thmtools,enumerate}
\usepackage{mathtools}

\usepackage[german,english]{babel}
\usepackage[autostyle]{csquotes}

\usepackage[all]{xy}
\usepackage{pstricks}

\usepackage{tikz-cd}
\usetikzlibrary{babel}

\usepackage{hyperref}
\hypersetup{
	colorlinks=true,
	linkcolor=red,
	citecolor=blue}

\theoremstyle{plain}

\newtheorem{thm}{Theorem}[section]
\newtheorem{pro}[thm]{Proposition}
\newtheorem{lem}[thm]{Lemma}
\newtheorem{cor}[thm]{Corollary}

\theoremstyle{definition}

\newtheorem{rem}[thm]{Remark}

\newcommand{\N}{\mathbb{N}}
\newcommand{\Q}{\mathbb{Q}}
\newcommand{\R}{\mathbb{R}}
\newcommand{\C}{\mathbb{C}}

\newcommand{\can}{\mathrm{can}}

\DeclareMathOperator{\Div}{Div}
\DeclareMathOperator{\Supp}{Supp}

\begin{document}
	
	\title[A few remarks on effectivity and good minimal models]{A few remarks on effectivity\\ and good minimal models}
	
	\author[V.\ Lazi\'c]{Vladimir Lazi\'c}
	\address{Fachrichtung Mathematik, Campus, Geb\"aude E2.4, Universit\"at des Saarlandes, 66123 Saarbr\"ucken, Germany}
	\email{lazic@math.uni-sb.de}

	\thanks{I gratefully acknowledge support by the Deutsche Forschungsgemeinschaft (DFG, German Research
Foundation) – Project-ID 286237555 – TRR 195. I would like to thank Junpeng Jiao, Fanjun Meng, Nikolaos Tsakanikas, Zhixin Xie and Ziquan Zhuang for many useful comments and suggestions.
		}
	
	\begin{abstract}
	We prove several results relating the nonvanishing and the existence of good minimal models of different pairs that have the same underlying variety.
	\end{abstract}

	\maketitle
	
	\begingroup
		\hypersetup{linkcolor=black}
		\setcounter{tocdepth}{1}
		\tableofcontents
	\endgroup
	
\section{Introduction}

The purpose of this short note is to address the following question: for a fixed normal projective variety $X$ over $\C$, if one knows the nonvanishing or the existence of good minimal models for \emph{one} log canonical pair $(X,\Delta)$, can one deduce the nonvanishing or the existence of good minimal models for \emph{another} log canonical pair $(X,\Delta')$ with $K_X+\Delta'$ pseudoeffective?

When the pair $(X,\Delta)$ is not of log general type, that is, if $K_X+\Delta$ is not a big $\R$-divisor, then much can be said under certain assumptions, as we will see below. 

\medskip

\noindent{\sc Effectivity.}
The first result is that effectivity is, in some sense, a closed condition on the set of log canonical pairs on a fixed variety.

\begin{thm}\label{thm:effectivethreshold}
Assume the existence of good minimal models for projective log canonical pairs in dimensions at most $n-1$. 

Let $(X,\Delta)$ be a projective log canonical pair of dimension $n$ such that $K_X+\Delta$ is pseudoeffective. Let $G\geq0$ be an $\R$-Cartier $\R$-divisor such that $0\leq\kappa_\iota(X,K_X+\Delta+G)<n$,\footnote{Here, $\kappa_\iota$ denotes the invariant Iitaka dimension, see Section \ref{sec:prelim}.} and set
$$\delta:=\inf\{t\geq0\mid \kappa_\iota(X,K_X+\Delta+tG)\geq0\}.$$
Then:
\begin{enumerate}[\normalfont (a)]
\item[(a)] $\kappa_\iota(X,K_X+\Delta+\delta G)\geq0$, and
\item[(b)] if $\kappa_\iota(X,K_X+\Delta+G)>0$, then $\delta=0$ and $(X,\Delta)$ has a good minimal model.
\end{enumerate}
\end{thm}

Note that in the theorem above I do not assume anything on the singularities of the pair $(X,\Delta+G)$. This result is inspired in part by \cite[Corollary 1.6]{Jia22}, where Jiao shows a version of Theorem \ref{thm:effectivethreshold}(a) under some stricter assumptions in lower dimensions and assuming that $(X,\Delta)$ is a klt pair such that $\Delta$ is a $\Q$-divisor.

\begin{rem}\label{rem:notpsef}
Theorem \ref{thm:effectivethreshold}(a) gives new information only when $K_X$ is pseudoeffective: indeed, when $K_X$ is not pseudoeffective, then for any $\R$-divisor $\Delta$ such that the pair $(X,\Delta)$ is log canonical and $K_X+\Delta$ is pseudoeffective, we have $\kappa_\iota(X,K_X+\Delta)\geq0$; this follows from the proof of \cite[Theorem 1.1]{LM21}.
\end{rem}

\medskip

\noindent{\sc Good models.}
The following result is an immediate corollary of Theorem \ref{thm:effectivethreshold}. It generalises \cite[Theorem 4.4]{Lai11}, where  it was proved for terminal varieties; \cite[Theorem 2.12]{HX13} and \cite[Theorem 1.3]{Hash18a}, where it was proved for klt pairs $(X,\Delta)$ such that $\Delta$ is a $\Q$-divisor; and \cite[Theorem 2.3]{Laz23}, where it was proved for log canonical pairs $(X,\Delta)$ such that $\Delta$ is a $\Q$-divisor.

\begin{thm}\label{thm:positivekodaira}
Assume the existence of good minimal models for projective log canonical pairs in dimension $n-1$.

Let $(X,\Delta)$ be a projective log canonical pair of dimension $n$. If $\kappa_\iota(X,K_X+\Delta)\geq1$, then $(X,\Delta)$ has a good minimal model.
\end{thm}

The following result complements Theorem \ref{thm:effectivethreshold} and is partly inspired by \cite[Theorem 1.4]{MZ23}. There, Meng and Zhuang show that if one knows the existence of a good minimal model for a log canonical pair $(X,\Delta)$, then one can deduce the existence of a good minimal model for a log canonical pair $(X,\Delta')$, assuming that $K_X+\Delta'$ is pseudoeffective, $\Delta'\leq\Delta$ and $\Supp\Delta=\Supp\Delta'$. Their very interesting argument uses previous results of Birkar and Hashizume--Hu to be able to run the MMP when one reduces the coefficients of a boundary. With appropriate assumptions in lower dimensions and by using different methods, one can improve their result considerably.

\begin{thm}\label{thm:main2}
Assume the existence of good minimal models for projective log canonical pairs in dimensions at most $n-1$.

Let $(X,\Delta)$ be a projective $\Q$-factorial log canonical pair of dimension $n$ such that $0\leq\kappa_\iota(X,K_X+\Delta)<n$. Assume that $(X,\Delta)$ has a good minimal model. Then:
\begin{enumerate}[\normalfont (a)]
\item if $\Delta'$ is an $\R$-divisor on $X$ such that the pair $(X,\Delta')$ is log canonical and $K_X+\Delta'$ is pseudoeffective, then $\kappa_\iota(X,K_X+\Delta')\geq0$,
\item if $\Delta'$ is an $\R$-divisor on $X$ such that the pair $(X,\Delta')$ is log canonical, $K_X+\Delta'$ is pseudoeffective and $\Delta'\leq\Delta$, then $(X,\Delta')$ has a good minimal model.
\end{enumerate}
If $K_X$ is additionally pseudoeffective, then:
\begin{enumerate}[\normalfont (i)]
\item if $\Delta'$ is an $\R$-divisor on $X$ such that the pair $(X,\Delta')$ is log canonical and $\Supp\Delta'\subseteq\Supp\Delta$, then $(X,\Delta')$ has a good minimal model,
\item if $\kappa_\iota(X,K_X+\Delta)>0$ and if $\Delta''$ is an $\R$-divisor on $X$ such that the pair $(X,\Delta'')$ is log canonical and $\Supp\Delta\subseteq\Supp\Delta''$, then $(X,\Delta'')$ has a good minimal model.
\end{enumerate}
\end{thm}

As an easy corollary, we obtain:

\begin{cor}\label{cor:onetoall}
Assume the existence of good minimal models for projective log canonical pairs in dimensions at most $n-1$ and the semiampleness part of the Abundance conjecture in dimension $n$.

Let $(X,\Delta)$ be a $\Q$-factorial log canonical pair of dimension $n$ such that $0\leq\kappa_\iota(X,K_X+\Delta)<n$. Then for any $\R$-divisor $\Delta'$ on $X$, if the pair $(X,\Delta')$ is log canonical and if $K_X+\Delta'$ is pseudoeffective, then $(X,\Delta')$ has a good minimal model.
\end{cor}

\section{Preliminaries}\label{sec:prelim}

I work over $\C$. Unless explicitly stated otherwise, all varieties in the paper are normal and projective. A \emph{fibration} is a projective surjective morphism with connected fibres. A \emph{birational contraction} is a birational map whose inverse does not contract any divisors. 

The standard reference for the definitions and basic results on the singularities of pairs and the Minimal Model Program is \cite{KM98}. A \emph{pair} $(X,\Delta)$ in this paper always has a boundary $\Delta$ which is an effective $\R$-divisor.

We will need the notion of \emph{dlt blowups} of a log canonical pair, see \cite[Theorem 10.5]{Fuj17}.

\subsection*{Good models}

Let $X$ and $Y$ be normal varieties, and let $D$ be an $\R$-Cartier $\R$-divisor on $X$. A birational contraction $f\colon X\dashrightarrow Y$ is a \emph{good minimal model for $D$}, or simply a \emph{good model for $D$}, if $f_*D$ is $\R$-Cartier and semiample, and if there exists a resolution of indeterminacies $(p,q)\colon W\to X\times Y$ of the map $f$ such that $p^*D=q^*f_*D+E$, where $E\geq0$ is a $q$-exceptional $\R$-divisor which contains the whole $q$-exceptional locus in its support.

This definition is the standard definition. A weaker notion is that of \emph{good models in the sense of Birkar--Hashizume}, see \cite[Section 2]{LM21} for details; it is used only occasionally in this paper. However, I need the following lemma: it says that the existence of good models and the existence of good models in the sense of Birkar--Hashizume are equivalent problems. Lemma \ref{lem:goodmodelsequivalence} is useful since many results in the literature show the existence of good models in the sense of Birkar--Hashizume.

\begin{lem}\label{lem:goodmodelsequivalence}
Let $(X,\Delta)$ be a projective log canonical pair such that $K_X+\Delta$ is pseudoeffective. If there exists a good model in the sense of Birkar--Hashizume of $(X,\Delta)$, then there exists a good model of $(X,\Delta)$.
\end{lem}

\begin{proof}
By \cite[Theorem 1.7]{HH20}, we may run a $(K_X+\Delta)$-MMP with scaling of an ample divisor which terminates with a minimal model $(Y,\Delta_Y)$. Let $(p,q)\colon W\to X\times Y$ be a resolution of indeterminacies of the map $X\dashrightarrow Y$ which is a log resolution of $(X,\Delta)$. By applying \cite[Lemma 2.15]{Has19} to the maps $p$ and $q$, we deduce that the pair $(Y,\Delta_Y)$ has a good model in the sense of Birkar--Hashizume.\footnote{Note that good models in the sense of Birkar--Hashizume are referred to as good minimal models in \cite{Has19}.} But now we conclude that $K_Y+\Delta_Y$ is semiample by the same argument as in the third paragraph of the proof of \cite[Lemma 4.1]{LM21}, hence $(Y,\Delta_Y)$ is a good model of $(X,\Delta)$.
\end{proof}

The following result is an immediate corollary of \cite[Lemma 2.15]{Has19} and of Lemma \ref{lem:goodmodelsequivalence}.

\begin{cor}\label{cor:upanddown}
Let $(X,\Delta)$ be a projective log canonical pair and let $f\colon Y\to X$ be a projective birational morphism such that there exist effective $\R$-Cartier $\R$-divisors $\Gamma$ and $E$ on $Y$, where $E$ is $f$-exceptional, with the property that the pair $(Y,\Gamma)$ is log canonical and 
$$K_Y+\Gamma\sim_\R f^*(K_X+\Delta)+E.$$
Then $(X,\Delta)$ has a good model if and only if $(Y,\Gamma)$ has a good model.
\end{cor}

\subsection*{Numerical dimension}

In this paper, if $X$ is a normal projective variety and $D$ is a pseudoeffective $\R$-Cartier $\R$-divisor on $X$, then $ \nu(X,D) $ denotes the \emph{numerical dimension} of $ D $, see \cite[Chapter V]{Nak04}, \cite{Kaw85}; this was denoted by $\kappa_\sigma$ in \cite{Nak04}.

I include the following important result for the lack of explicit reference.

\begin{thm}\label{thm:nu=0}
Let $(X,\Delta)$ be a projective log canonical pair such that\linebreak $\nu(X,K_X+\Delta)=0$. Then there exists a $(K_X+\Delta)$-MMP with scaling of an ample divisor which terminates with a good model of $(X,\Delta)$.
\end{thm}

\begin{proof}
If $(X,\Delta)$ is a klt pair, this was shown in \cite{Dru11}, and if $(X,\Delta)$ is a dlt pair, the result was proved in \cite{Gon11}. In the generality given here, it is a special case of \cite[Corollary 5.7]{Has23}, but here I give a short alternative proof.

Let $\pi\colon (X',\Delta')\to (X,\Delta)$ be a dlt blowup of $(X,\Delta)$. Then $\nu(X',K_{X'}+\Delta')=0$ by Remark \ref{rem:dimensions}(a) below, hence $(X',\Delta')$ has a good model by \cite[Theorems 5.1 and 6.1]{Gon11}. Thus, $(X,\Delta)$ has a good model by Corollary \ref{cor:upanddown}. Finally, the statement on the existence of the MMP with scaling follows from \cite[Theorem 1.7]{HH20} and \cite[Lemma 6.1]{Gon11}.
\end{proof}

\subsection*{Invariant Iitaka dimension}

I use the \emph{invariant Iitaka dimension} of a pseudoeffective $\R$-Cartier $\R$-divisor $D$ on a normal projective variety $X$, denoted by $\kappa_\iota(X,D)$ and introduced in \cite{Cho08}; if the divisor $D$ has rational coefficients or if $D\geq0$, its Iitaka dimension is denoted by $\kappa(X,D)$. The invariant Iitaka dimension shares many good properties with the Iitaka dimension of $\Q$-divisors \cite[Section 2.5]{Fuj17}, and this fact will be used throughout without explicit mention. 

\begin{rem}\label{rem:dimensions}
We frequently need the following properties.
\begin{enumerate}[\normalfont (a)]
\item If $D$ is an $\R$-Cartier $\R$-divisor on a normal projective variety $X$, if $f\colon Y\to X$ is a birational morphism from a normal projective variety $Y$, and if $E$ is an effective $f$-exceptional $\R$-Cartier $\R$-divisor on $Y$, then
$$\kappa_\iota(X,D)=\kappa_\iota(Y,f^*D+E)\quad\text{and}\quad\nu(X,D)=\nu(Y,f^*D+E);$$
see for instance \cite[\S2.2]{LP20a} for references and discussion.
\item If $D_1$ and $D_2$ are effective $\R$-Cartier $\R$-divisors on a normal projective variety $X$ such that $\Supp D_1=\Supp D_2$, then $\kappa_\iota(X,D_1)=\kappa_\iota(X,D_2)$ and $\nu(X,D_1)=\nu(X,D_2)$; the proof is easy and the same as that of \cite[Lemma 2.9]{DL15}. Moreover, the same proof shows that if $\Supp D_1\subseteq\Supp D_2$, then $\kappa_\iota(X,D_1)\leq\kappa_\iota(X,D_2)$ and $\nu(X,D_1)\leq\nu(X,D_2)$.
\end{enumerate}
\end{rem}

The following result \cite[Lemma 2.3]{LP18} is crucial for this paper.

\begin{lem}\label{lem:iitaka}
Let $X$ be a normal projective variety and let $L$ be a $\Q$-Cartier $\Q$-divisor on $X$ with $\kappa(X,L)\geq0$. Then for any sufficiently high resolution $\pi\colon Y\to X$ there exists a fibration $f\colon Y\to Z$:
\[
\xymatrix{ 
Y \ar[d]^{\pi} \ar[r]^{f} & Z \\
X  & 
}
\]
such that $\dim Z=\kappa(X,L)$, and for every $\pi$-exceptional $\Q$-divisor $E\geq0$ on $Y$ and for a very general fibre $F$ of $f$ we have
$$\kappa\big(F,(\pi^*L+E)|_F\big)=0.$$
\end{lem}

\begin{proof}
The formulation is slightly different than that of \cite[Lemma 2.3]{LP18}: here, $\pi$ is any sufficiently high resolution; in fact, it was implicitly used in \cite{LP18} in this form. The proof is the same as that of \cite[Lemma 2.3]{LP18}: the only thing to notice is that once one fixes the variety $X_\infty$ in the proof of \cite[Theorem 2.1.33]{Laz04}, one can freely replace it by any of its higher birational models.
\end{proof}

The following easy result does not seem to have a proper reference and will be used several times in this paper.

\begin{lem}\label{lem:invariantiitaka}
Let $X$ be a $\Q$-factorial projective variety, and let $D_1$ and $D_2$ be $\R$-divisors on $X$ such that $D_1\sim_\R D_2$.
\begin{enumerate}[\normalfont (a)]
\item If $D_1$ and $D_2$ are $\Q$-divisors, then $D_1\sim_\Q D_2$.
\item If $D_1\geq0$, $D_2\geq0$ and $\kappa(X,D_1)=0$, then $D_1=D_2$.
\end{enumerate}
\end{lem}

\begin{proof}
The proof is similar to that of \cite[Lemma 2.3]{CL12a}.

I first show (a). There exist real numbers $r_1,\dots,r_k$ and rational functions $f_1,\dots,f_k\in k(X)$ such that $D_2=D_1+\sum_{i=1}^k r_i (f_i)$. The system of linear equations $D_2=D_1+\sum_{i=1}^k x_i (f_i)$ with rational coefficients has a real solution $(r_1,\dots,r_k)$, hence it has a rational solution, which gives (a).

Next I show (b). Let $D_1'$ be a $\Q$-divisor such that $D_1'\geq D_1$ and $\Supp D_1'=\Supp D_1$. Then $D_1'\sim_\R D_2+(D_1'-D_1)$ and $\kappa(X,D_1')=0$ by Remark \ref{rem:dimensions}(b). Thus, by replacing $D_1$ by $D_1'$ and $D_2$ by $D_2+(D_1'-D_1)$, we may assume that $D_1$ is a $\Q$-divisor.

There exist real numbers $r_1,\dots,r_k$ and rational functions $f_1,\dots,f_k\in k(X)$ such that $D_2=D_1+\sum_{i=1}^k r_i (f_i)$. Let $W\subseteq \Div_{\mathbb R}(X)$ be the subspace spanned by the components of $D_1$ and of all $(f_i)$. Let $W_0\subseteq W$ be the subspace of divisors which are $\R$-linear combinations of the principal divisors $(f_1),\dots,(f_k)$. Then $W_0$ is a rational subspace of $W$, and consider the quotient map $\pi\colon W\to W/W_0$. Then the set
$$V:=\{G\in \pi^{-1}(\pi(D_1))\mid G\geq 0\}$$
is not empty as it contains $D_1$ and $D_2$, and it is a rational affine subspace of $W$ since $D_1$ is a $\Q$-divisor. Thus, there exist $\Q$-divisors $G_1,\dots,G_\ell\in V$ and positive real numbers $t_1,\dots,t_\ell$ such that 
$$D_2=\sum_{i=1}^\ell t_iG_i\quad\text{and}\quad \sum_{i=1}^\ell t_i=1;$$
in particular, we have $D_1\sim_\R G_i$ for every $i$. Part (a) yields $D_1\sim_\Q G_i$ for every $i$, hence $D_1=G_i$ for each $i$ since $\kappa(X,D_1)=0$ and since $D_1$ and all $G_i$ are $\Q$-divisors. Therefore, $D_2=\sum_{i=1}^\ell t_iD_1=D_1$, as desired.
\end{proof}

\section{Proofs of the main results}

I start with the following easy but important lemma.

\begin{lem}\label{lem:jiao}
Let $X$ be a $\Q$-factorial projective variety and let $H$ be an effective $\R$-divisor on $X$. Let $G$ be an $\R$-divisor on $X$ such that there exists a strictly decreasing sequence of real numbers $(\varepsilon_n)_{n\in\N}$ such that $\lim\limits_{n\to\infty}\varepsilon_n=0$ and $\kappa_\iota(X,G+\varepsilon_n H)=0$. Then $\kappa_\iota(X,G)=0$.
\end{lem}

\begin{proof}
The proof is implicit in the proof of \cite[Corollary 1.6]{Jia22} and in Step 3 of the proof of \cite[Theorem 3.1]{LP18}, and here I provide the details.

For each $n\in\N$, let $D_n$ be an effective divisor $\R$-linearly equivalent to $G+\varepsilon_n H$. Then for any positive integer $k$ we have
$$D_k+(\varepsilon_1-\varepsilon_k)H\sim_\R D_1.$$
This implies
$$D_k+(\varepsilon_1-\varepsilon_k)H=D_1$$
by Lemma \ref{lem:invariantiitaka}(b) since $\kappa(X,D_1)=0$, and in particular,
$$D_1\geq(\varepsilon_1-\varepsilon_k)H.$$
Letting $k\to\infty$ we obtain $D_1\geq \varepsilon_1 H$. Since $G\sim_\R D_1-\varepsilon_1 H\geq0$, this gives $\kappa_\iota(X,G)\geq0$. On the other hand, we obviously have $\kappa_\iota(X,G)\leq\kappa_\iota(X,D_1)=0$, and thus $\kappa_\iota(X,G)=0$.
\end{proof}

The following result generalises \cite[Proposition 3.2]{LP18} to pairs with log canonical singularities.

\begin{pro}\label{pro:contraction}
Assume the existence of good models for projective log canonical pairs in dimensions at most $n-1$.

Let $(X,\Delta)$ be a projective log canonical pair such that $K_X+\Delta$ is pseudoeffective. If there exists a fibration $X\to Z$ to a normal projective variety $Z$ such that $\dim Z\geq 1$ and $K_X+\Delta$ is not big over $Z$, then $(X,\Delta)$ has a good model.
\end{pro}

\begin{proof}
The divisor $K_X+\Delta$ is effective over $Z$ by induction on the dimension and by \cite[Lemma 3.2.1]{BCHM}. By \cite[Theorem 1.2]{HH20} there exists a good model in the sense of Birkar--Hashizume $(X,\Delta)\dashrightarrow (X_{\min},\Delta_{\min})$ of $(X,\Delta)$ over $Z$. In particular, we have
\begin{equation}\label{eq:a}
\kappa_\iota(X,K_X+\Delta)=\kappa_\iota(X_{\min},K_{X_{\min}}+\Delta_{\min})
\end{equation}
and
\begin{equation}\label{eq:b}
\nu(X,K_X+\Delta)=\nu(X_{\min},K_{X_{\min}}+\Delta_{\min})
\end{equation}
by \cite[Lemma 2.6]{LM21}.\footnote{Note that \cite[Lemma 2.6]{LM21} is stated in the absolute setting, but the proof works in the relative setting treated here.} Furthermore, there exists a fibration $\varphi\colon X_{\min}\to X_\can$ over $Z$ such that
\begin{equation}\label{eq:d}
K_{X_{\min}}+\Delta_{\min}\sim_\R\varphi^*A
\end{equation}
for some $\R$-divisor $A$ on $X_\can$ which is ample over $Z$.
\[
\xymatrix{ 
X \ar[dr] \ar@{-->}[r] & X_{\min} \ar[d] \ar[r]^{\varphi} & X_\can \ar[dl]\\
& Z & 
}
\]
Since $K_X+\Delta$ is not big over $Z$, the divisor $K_{X_{\min}}+\Delta_{\min}$ is not big over $Z$; this follows, for instance, by applying \cite[Lemma 2.6]{LM21} to general fibres of the morphisms $X\to Z$ and $X_{\min}\to Z$. Therefore, $\dim X_\can<\dim X$. By \cite[Theorem 1.5]{Has19} and by \eqref{eq:d}, there exists a good model in the sense of Birkar--Hashizume of $(X_{\min},\Delta_{\min})$, hence
\begin{equation}\label{eq:c}
\kappa_\iota(X_{\min},K_{X_{\min}}+\Delta_{\min})=\nu(X_{\min},K_{X_{\min}}+\Delta_{\min})
\end{equation}
by \cite[Lemma 2.6]{LM21}. But then the pair $(X,\Delta)$ has a good model by \eqref{eq:a}, \eqref{eq:b}, \eqref{eq:c} and by \cite[Lemma 4.1]{LM21}. This proves the proposition.
\end{proof}

Now I can prove the results announced in the introduction.

\begin{proof}[Proof of Theorem \ref{thm:effectivethreshold}]
Let $\rho\colon W\to X$ be a log resolution of $X$. Then there exist effective $\R$-divisors $\Delta_W$ and $E_E$ without common components such that
$$K_W+\Delta_W\sim_\R\rho^*(K_X+\Delta)+E_W.$$
Note that 
$$\kappa_\iota(X,K_X+\Delta+G)=\kappa_\iota(W,K_W+\Delta_W+\rho^*G)$$
and
$$\kappa_\iota(X,K_X+\Delta+\delta G)=\kappa_\iota(W,K_W+\Delta_W+\delta\rho^*G)$$
by Remark \ref{rem:dimensions}(a), and the pair $(X,\Delta)$ has a good model if and only if the pair $(W,\Delta_W)$ has a good model by Corollary \ref{cor:upanddown}. Therefore, by replacing $(X,\Delta)$ by $(W,\Delta_W)$ and $G$ by $\rho^*G$, we may assume that $(X,\Delta)$ is log smooth.

If $\kappa_\iota(X,K_X+\Delta+G)=0$, then $\kappa_\iota(X,K_X+\Delta+\delta G)=0$ by Lemma \ref{lem:jiao}. Therefore, I assume for the remainder of the proof that $\kappa_\iota(X,K_X+\Delta+G)>0$. The argument is inspired by the proof of \cite[Theorem 3.3]{LP18}.

Let $D\geq0$ be an $\R$-divisor such that $K_X+\Delta+G\sim_\R D$. Pick an $\R$-divisor $D'\geq0$ such that $\Supp D'=\Supp D$ and such that $D+D'$ is a $\Q$-divisor. We have
\begin{equation}\label{eq:5}
K_X+\Delta+G+D'\sim_\R D+D'
\end{equation}
and thus
$$0<\kappa(X,D+D')<n$$
by Remark \ref{rem:dimensions}(b). By Lemma \ref{lem:iitaka} there exists a log resolution $\pi\colon Y\to X$ of the pair $(X,\Delta)$ and a fibration $f\colon Y\to Z$:
\[
\xymatrix{ 
Y \ar[d]_{\pi} \ar[r]^{f} & Z\\
X & 
}
\]
such that $\dim Z=\kappa(X,D+D')$ and for a very general fibre $F$ of $f$ and for every $\pi$-exceptional $\Q$-divisor $H$ on $Y$ we have
\begin{equation}\label{eq:exceptional2}
\kappa\big(F,\big(\pi^*(D+D')+H\big)|_F\big)=0.
\end{equation}
There exist effective $\R$-divisors $\Gamma$ and $E$ without common components such that
$$K_Y+\Gamma\sim_\R\pi^*(K_X+\Delta)+E;$$
thus, $(Y,\Gamma)$ is a log smooth log canonical pair such that $K_Y+\Gamma$ is pseudoeffective. By \eqref{eq:5} and \eqref{eq:exceptional2} we have
\begin{align}
\kappa_\iota\big(F,(K_Y&+\Gamma+\pi^*G+\pi^*D'+\lceil E\rceil-E)|_F\big)\label{eq:Gamma2}\\
&=\kappa_\iota\big(F,\big(\pi^*(K_X+\Delta+G+D')+\lceil E\rceil\big)|_F\big) \notag\\
&=\kappa\big(F,\big(\pi^*(D+D')+\lceil E\rceil\big)|_F\big)=0.\notag
\end{align}
Since $\pi^*G+\pi^*D'+\lceil E\rceil-E\geq0$ and $\kappa_\iota\big(F,(K_Y+\Gamma)|_F\big)\geq0$ by induction on the dimension, the equation \eqref{eq:Gamma2} implies
$$\kappa_\iota\big(F,(K_Y+\Gamma)|_F\big)=0.$$ 
But then $(Y,\Gamma)$ has a good model by Proposition \ref{pro:contraction}, thus $(X,\Delta)$ has a good model by Corollary \ref{cor:upanddown}, which proves (b).
\end{proof}

\begin{proof}[Proof of Theorem \ref{thm:positivekodaira}]
If $\kappa_\iota(X,K_X+\Delta)=n$, then we conclude by \cite[Theorem 1.2]{HH20}. If $0<\kappa_\iota(X,K_X+\Delta)<n$, then we conclude by Theorem \ref{thm:effectivethreshold}(b), by setting $G=0$ in that statement.
\end{proof}

\begin{proof}[Proof of Theorem \ref{thm:main2}]
Note first that (a) follows from (b). Indeed, if $K_X$ is pseudoeffective, then (b) gives in particular that $\kappa(X,K_X)\geq0$, which yields (a). If $K_X$ is not pseudoeffective, then (a) follows from Remark \ref{rem:notpsef}.

\medskip

For (b), if $\kappa_\iota(X,K_X+\Delta)>0$, then we conclude by Theorem \ref{thm:effectivethreshold}(b). If $\kappa_\iota(X,K_X+\Delta)=0$, then $\nu(X,K_X+\Delta)=0$ as the pair $(X,\Delta)$ has a good model. Since $\nu(X,K_X+\Delta')\leq\nu(X,K_X+\Delta)=0$, we have $\nu(X,K_X+\Delta')=0$, and we conclude by Theorem \ref{thm:nu=0}.

\medskip

For (i), let $t$ be a positive real number such that $t\Delta'\leq\Delta$. By (b), the pair $(X,t\Delta')$ has a good model; in particular,
\begin{equation}\label{eq:4}
\kappa_\iota(X,K_X+t\Delta')=\nu(X,K_X+t\Delta').
\end{equation}
By (a) there exists an effective $\Q$-divisor $D$ such that $K_X\sim_\Q D$, hence 
$$K_X+t\Delta'\sim_\R D+t\Delta',\quad K_X+\Delta'\sim_\R D+\Delta',$$
and
$$\Supp(D+t\Delta')=\Supp(D+\Delta').$$
This, together with \eqref{eq:4} and Remark \ref{rem:dimensions}(b), implies that $\kappa_\iota(X,K_X+\Delta')=\nu(X,K_X+\Delta')$, hence $(X,\Delta')$ has a good model by \cite[Lemma 4.1]{LM21}.

\medskip

Finally, we show (ii). By (a) there exists an effective $\Q$-divisor $D$ such that $K_X\sim_\Q D$, hence 
$$K_X+\Delta\sim_\R D+\Delta,\quad K_X+\Delta''\sim_\R D+\Delta'',$$
and
$$\Supp(D+\Delta)\subseteq\Supp(D+\Delta'').$$
Together with Remark \ref{rem:dimensions}(b) this implies that 
$$\kappa_\iota(X,K_X+\Delta'')\geq\kappa_\iota(X,K_X+\Delta)>0,$$
and we conclude by Theorem \ref{thm:positivekodaira}.
\end{proof}

\begin{proof}[Proof of Corollary \ref{cor:onetoall}]
By \cite[Theorem B]{LT22} we may run a $(K_X+\Delta)$-MMP which terminates with a minimal model $(Y,\Delta_Y)$. Since we assume the semiampleness part of the Abundance conjecture, the divisor $K_Y+\Delta_Y$ is semiample and therefore, the pair $(Y,\Delta_Y)$ is a good model of $(X,\Delta)$.

If now $\Delta'$ is an $\R$-divisor on $X$ such that the pair $(X,\Delta')$ is log canonical and $K_X+\Delta'$ is pseudoeffective, then $\kappa_\iota(X,K_X+\Delta')\geq0$ by Theorem \ref{thm:main2}(a). By the same argument as in the previous paragraph applied to the pair $(X,\Delta')$ we conclude that $(X,\Delta')$ has a good model.
\end{proof}
	
	\bibliographystyle{amsalpha}
	\bibliography{biblio}

\providecommand{\bysame}{\leavevmode\hbox to3em{\hrulefill}\thinspace}
\providecommand{\MR}{\relax\ifhmode\unskip\space\fi MR }
\providecommand{\MRhref}[2]{%
  \href{http://www.ams.org/mathscinet-getitem?mr=#1}{#2}
}
\providecommand{\href}[2]{#2}
\begin{thebibliography}{BCHM10}

\bibitem[BCHM10]{BCHM}
C.~Birkar, P.~Cascini, C.~D. Hacon, and J.~M\textsuperscript{c}Kernan,
  \emph{Existence of minimal models for varieties of log general type}, J.
  Amer. Math. Soc. \textbf{23} (2010), no.~2, 405--468.

\bibitem[Cho08]{Cho08}
S.~R. Choi, \emph{{The geography of log models and its applications}}, PhD
  Thesis, Johns Hopkins University, 2008, available at
  \url{https://math.jhu.edu/webarchive/grad/ChoiThesis.pdf}\setbox0=\hbox{2008}.

\bibitem[CL12]{CL12a}
P.~Cascini and V.~Lazi{\'c}, \emph{New outlook on the {M}inimal {M}odel
  {P}rogram, {I}}, Duke Math. J. \textbf{161} (2012), no.~12, 2415--2467.

\bibitem[DL15]{DL15}
T.~Dorsch and V.~Lazi\'c, \emph{A note on the abundance conjecture}, Algebraic
  Geometry \textbf{2} (2015), no.~4, 476--488.

\bibitem[Dru11]{Dru11}
S.~Druel, \emph{Quelques remarques sur la d\'ecomposition de {Z}ariski
  divisorielle sur les vari\'et\'es dont la premi\`ere classe de {C}hern est
  nulle}, Math. Z. \textbf{267} (2011), no.~1-2, 413--423.

\bibitem[Fuj17]{Fuj17}
O.~Fujino, \emph{Foundations of the minimal model program}, MSJ Memoirs,
  vol.~35, Mathematical Society of Japan, Tokyo, 2017.

\bibitem[Gon11]{Gon11}
Y.~Gongyo, \emph{On the minimal model theory for dlt pairs of numerical log
  kodaira dimension zero}, Math. Res. Lett. \textbf{18} (2011), no.~5,
  991--1000.

\bibitem[Has18]{Hash18a}
K.~Hashizume, \emph{Minimal model theory for relatively trivial log canonical
  pairs}, Ann. Inst. Fourier \textbf{68} (2018), no.~5, 2069--2107.

\bibitem[Has19]{Has19}
\bysame, \emph{Remarks on special kinds of the relative log minimal model
  program}, Manuscripta Math. \textbf{160} (2019), no.~3-4, 285--314.

\bibitem[Has23]{Has23}
\bysame, \emph{Minimal model program for normal pairs along log canonical
  locus}, arXiv:2310.13904\setbox0=\hbox{2023}.

\bibitem[HH20]{HH20}
K.~Hashizume and Z.-Y. Hu, \emph{On minimal model theory for log abundant lc
  pairs}, J. Reine Angew. Math. \textbf{767} (2020), 109--159.

\bibitem[HX13]{HX13}
C.~D. Hacon and C.~Xu, \emph{Existence of log canonical closures}, Invent.
  Math. \textbf{192} (2013), no.~1, 161--195.

\bibitem[Jia22]{Jia22}
J.~Jiao, \emph{On the finiteness of ample models}, Math. Res. Lett. \textbf{29}
  (2022), no.~3, 763--783.

\bibitem[Kaw85]{Kaw85}
Y.~Kawamata, \emph{Pluricanonical systems on minimal algebraic varieties},
  Invent. Math. \textbf{79} (1985), 567--588.

\bibitem[KM98]{KM98}
J.~Koll{\'a}r and S.~Mori, \emph{Birational geometry of algebraic varieties},
  Cambridge Tracts in Mathematics, vol. 134, Cambridge University Press,
  Cambridge, 1998.

\bibitem[Lai11]{Lai11}
C.-J. Lai, \emph{Varieties fibered by good minimal models}, Math. Ann.
  \textbf{350} (2011), no.~3, 533--547.

\bibitem[Laz04]{Laz04}
R.~Lazarsfeld, \emph{Positivity in algebraic geometry. {I}, {II}}, Ergebnisse
  der Mathematik und ihrer Grenzgebiete, vol. 48, 49, Springer-Verlag, Berlin,
  2004.

\bibitem[Laz23]{Laz23}
V.~Lazi\'c, \emph{{Abundance for uniruled pairs which are not rationally
  connected}}, Enseign. Math. (2023), \url{https://doi.org/10.4171/LEM/1065}.

\bibitem[LM21]{LM21}
V.~Lazi\'{c} and F.~Meng, \emph{On nonvanishing for uniruled log canonical
  pairs}, Electron. Res. Arch. \textbf{29} (2021), no.~5, 3297--3308.

\bibitem[LP18]{LP18}
V.~Lazi\'c and Th. Peternell, \emph{Abundance for varieties with many
  differential forms}, \'Epijournal Geom. Alg\'ebrique \textbf{2} (2018),
  Article 1.

\bibitem[LP20]{LP20a}
\bysame, \emph{{On Generalised Abundance, I}}, Publ. Res. Inst. Math. Sci.
  \textbf{56} (2020), no.~2, 353--389.

\bibitem[LT22]{LT22}
V.~Lazi\'{c} and N.~Tsakanikas, \emph{On the existence of minimal models for
  log canonical pairs}, Publ. Res. Inst. Math. Sci. \textbf{58} (2022), no.~2,
  311--339.

\bibitem[MZ23]{MZ23}
F.~Meng and Z.~Zhuang, \emph{{MMP for locally stable families and wall crossing
  for moduli of stable pairs}}, arXiv:2311.01319\setbox0=\hbox{2023}.

\bibitem[Nak04]{Nak04}
N.~Nakayama, \emph{Zariski-decomposition and abundance}, MSJ Memoirs, vol.~14,
  Mathematical Society of Japan, Tokyo, 2004.

\end{thebibliography}
	
\end{document}